\documentclass[reqno,11pt]{amsart}

\usepackage{a4wide}
\usepackage{color}
\usepackage{mathrsfs}
\usepackage{mathtools}
\usepackage{amsmath}
\usepackage{amssymb}
\usepackage{bbm}
\usepackage{esint}
\numberwithin{equation}{section}
\usepackage[colorlinks,citecolor=green,linkcolor=red]{hyperref}

\usepackage[latin1]{inputenc}

\newcommand{\1}{\mathbbm 1}
\newcommand{\nchi}{{\raise.3ex\hbox{\(\chi\)}}}
\newcommand{\N}{\mathbb{N}}

\newcommand{\R}{\mathbb{R}}
\newcommand{\sfd}{{\sf d}}
\renewcommand{\d}{{\mathrm d}}
\newcommand{\X}{{\rm X}}

\newcommand{\mm}{\mathfrak{m}}
\newcommand{\lip}{{\rm lip}}

\newcommand{\FV}{{\rm FV}}

\newcommand{\fr}{\penalty-20\null\hfill\(\blacksquare\)}

\newtheorem{theorem}{Theorem}[section]

\newtheorem{lemma}[theorem]{Lemma}
\newtheorem{proposition}[theorem]{Proposition}
\newtheorem{definition}[theorem]{Definition}
\newtheorem{remark}[theorem]{Remark}

\begin{document}

\title[Decomposition of functions of bounded variation in Ahlfors-regular PI spaces]{Decomposition of functions of bounded variation \\ in Ahlfors-regular PI spaces}

\author[Paolo Bonicatto]{Paolo Bonicatto}
\address{Dipartimento di Matematica, Universit\`{a} di Trento, Via Sommarive 14, 38123 Trento, Italy}
\email{paolo.bonicatto@unitn.it}

\author[Panu Lahti]{Panu Lahti}
\address{Academy of Mathematics and Systems Science, Chinese Academy of Sciences, Beijing 100190, PR China}
\email{panulahti@amss.ac.cn}

\author[Enrico Pasqualetto]{Enrico Pasqualetto}
\address{Department of Mathematics and Statistics, P.O.\ Box 35, FI-40014 University of Jyvaskyla}
\email{enrico.e.pasqualetto@jyu.fi}


\date{\today}
\maketitle

\begin{center}
\small\textit{Dedicated to Nageswari Shanmugalingam on the occasion of her 60th birthday}
\end{center}
\begin{abstract}
We prove that, on a non-compact Ahlfors-regular PI space, every function of bounded variation $f$ can be decomposed
as a countable sum $\sum_{n=1}^\infty f_n$ of monotone functions of bounded variation $f_n$, in such a way that the
total variation of $f$ is the sum of the total variations of the functions $f_n$. Our result extends a Euclidean
result of Bianchini and Tonon to the metric setting.
\end{abstract}
\section{Introduction}
One of the basic principles in geometric measure theory is that functions of bounded variation can be studied through the geometry of their level sets. The Coarea formula makes this principle precise: the total variation measure of a BV function is obtained by integrating the perimeter measures of its superlevel sets. Thus, structural information on sets of finite perimeter can, at least in principle, be transferred to BV functions.

The purpose of this paper is to exploit this principle in the setting of PI spaces, by which we mean \emph{doubling}
metric measure spaces supporting a \emph{weak \((1,1)\)-Poincar\'{e} inequality} \cite{HKST}. Starting with \cite{Miranda},
the theory of BV functions and sets of finite perimeter has been developed in the metric setting, both in PI spaces (see,
e.g.\ \cite{Ambrosio01,Ambrosio02,KKST,Lahti20}, to name but a few) and in general metric measure spaces (see, e.g.,
\cite{AmbrosioDiMarino,AmbPasSha}).
Relying on the decomposition of sets of finite perimeter into \emph{indecomposable components} established in
\cite{ACMM} and \cite{BPR,L-decom} for Euclidean and PI spaces, respectively,
we prove that, on a non-compact Ahlfors regular PI space, every function of finite variation (with enough integrability) can be decomposed into a countable sum of \emph{monotone} BV functions, in such a way that both the function and its variation measure are decomposed without loss.

Here, monotonicity is understood in a measure-theoretic sense: a BV function $u$ is called monotone if, for almost every $t\in\mathbb R$, both $\{u>t\}$ and $\{u\leq t\}$ are indecomposable sets of finite perimeter. This notion is a natural
geometric measure theory
counterpart of topological monotonicity.
In the real line, the classical Jordan decomposition asserts that every BV function is the difference of two (ordinary) monotone functions. In higher dimensions, where there is no analogous order structure, connectedness or indecomposability properties of level sets provide the appropriate replacement. 

Let us state our main theorem.
\begin{theorem}\label{thm:main}
Let \((\X,\sfd,\mm)\) be an Ahlfors \(q\)-regular PI space with \(q>1\) and \(\mm(\X)=\infty\).
Let \(f\in{\rm FV}(\X)\) be given. Then there exists a sequence of monotone functions
\((f_n)_{n\in\N}\subseteq\FV(\X)\) such that
\[
f=\sum_{n=1}^\infty f_n\,\text{ absolutely in }L^{q/(q-1)}(\X)
\]
and
\[
|Df|=\sum_{n=1}^{\infty}|Df_n|\quad\text{ as finite Borel measures on }\X.
\]
\end{theorem}

In the above statement, \({\rm FV}(\X)\) denotes the space of all functions \(f\in L^{q/(q-1)}(\X)\)
whose total variation \(|Df|(\X)\) is finite; see \eqref{eq:def_FV}. We recall
that a series \(v=\sum_{n=1}^\infty v_n\) in a Banach space is \emph{absolutely convergent} if
\(\sum_{n=1}^\infty\|v_n\|<+\infty\) and \(\|\sum_{n=1}^N v_n-v\|\to 0\) as \(N\to\infty\).

The last identity is an essential part of the statement. It says that the decomposition creates no cancellation at the level of the variation. Thus the functions $f_n$ are not merely summable in a Banach-space sense but they also provide a splitting of the variation measure of $f$.

The proof combines two main ingredients. The first is the decomposition theory for sets of finite perimeter in PI spaces. If $E$ is a set of finite perimeter, then $E$ admits an essentially unique decomposition into maximal indecomposable components $(E_i)_i$ and it holds
\[
    P(E,\cdot)=\sum_i P(E_i,\cdot).
\]
This allows us to decompose almost every superlevel set $\{f>t\}$ into its essential connected components. The terminology is not merely suggestive: under an additional assumption (called the \emph{two-sidedness property}) on the space, we recently proved that indecomposable sets are precisely the connected sets, up to the choice of a suitable representative,
in the \emph{1-fine topology} and,
moreover, the essential connected components coincide with the $1$-fine connected components, see the related work~\cite{BLP}.

The second ingredient is the Coarea formula. Since
\[
    |Df|=\int_{\mathbb R}P(\{f>t\})\,\mathrm dt,
\]
a decomposition of the superlevel sets which is compatible as $t$ varies can be integrated in the level parameter to obtain a decomposition of the function itself. The main technical difficulty is precisely this compatibility: the essential connected components at different levels do not form an a priori ordered family. We therefore introduce an abstract class of nested families of sets and associate with each such family a BV function whose superlevel sets are the prescribed ones.

The key idea of the proof is to show that from a function $f\neq 0$ with finite total variation one can
extract a non-zero monotone function $g$ such that $|Df|=|D(f-g)|+|Dg|$.
The function \(g\) is constructed from a suitable family of indecomposable
components of the superlevel sets of \(f\), which are then saturated to
ensure monotonicity. Similarly to the Euclidean argument underlying one of
the proofs in~\cite{BG-structure}, an abstract decomposition procedure in
a Banach space allows this extraction to be iterated and yields
the desired countable decomposition.
\medskip

In the Euclidean setting, decompositions into
monotone components and the geometry of level sets of BV or Lipschitz functions have proved useful in the analysis of transport and continuity equations. The work of Alberti, Bianchini and Crippa~\cite{ABC-structure,ABC-continuity} introduced a geometric analysis of the level sets of Lipschitz Hamiltonians and used it to study uniqueness for the continuity equation. Related ideas were subsequently developed for steady and autonomous nearly incompressible vector fields, see ~\cite{BG-steady, BBG}. More recently, after establishing a monotone decomposition for weakly differentiable functions \cite{GK-decomposition}, Gusev and Korobkov used it to study chain-rule and renormalization properties and uniqueness for the continuity equation in connection with Nelson's conjecture~\cite{GK-Nelson}. Finally, a general decomposition result for BV functions on Euclidean spaces was obtained by Bianchini and Tonon \cite{BT},
and later revisited in \cite{BG-structure}.

We emphasize that the proof is entirely metric. It uses only the decomposability of sets of finite perimeter, the Coarea formula and the isoperimetric inequality of the ambient space. The argument therefore works at the level of generality naturally associated with PI spaces, where a robust first-order calculus and a rich theory of BV functions are available, but no smooth structure is assumed.
\subsection*{Acknowledgements}
E.P.\ was supported by the Research Council of Finland grant 362898. This work was supported by the Italian Ministry of University and Research (MUR) through the FIS2 project BRIDGE: ``Beyond Regularity In Differential Geometric Evolutions", project code FIS-2023-02224 (CUP: E53C25001790001).
\subsection*{AI usage statement}
No AI tools were used in the preparation of this manuscript.
\section{Functions of bounded variation in metric measure spaces}
\subsection{Definition via relaxation and basic properties}
In this paper, by a metric measure space \((\X,\sfd,\mm)\) we mean a complete separable metric space \((\X,\sfd)\)
together with a Borel measure \(\mm\) on \(\X\) that is boundedly finite, i.e.\ \(\mm(B)<\infty\)
for every bounded Borel set \(B\subseteq\X\). Given an open set \(\Omega\subseteq\X\), we denote by
\({\rm LIP}_{loc}(\Omega)\) the set of all locally-Lipschitz functions \(f\colon\Omega\to\R\),
and by \({\rm LIP}_{bs}(\Omega)\) the set of all boundedly-supported Lipschitz functions \(f\colon\Omega\to\R\).
For any \(f\in{\rm LIP}_{loc}(\Omega)\), we define the asymptotic slope function
\(\lip_a(f)\colon\Omega\to[0,\infty)\) as
\[
\lip_a(f)(x)\coloneqq\varlimsup_{y,z\to x}\frac{|f(y)-f(z)|}{\sfd(y,z)}\quad\text{ for every }x\in\X.
\]
It can be readily checked that \(\lip_a(f)\) is upper semicontinuous, thus Borel measurable.
\medskip

Following \cite{Miranda}, we consider the following definition of a metric BV function.
\begin{definition}[Function of bounded variation]
Let \((\X,\sfd,\mm)\) be a metric measure space. Given \(f\in L^1_{loc}(\X)\) and an open
set \(\Omega\subseteq\X\), we define the total variation of \(f\) on \(\Omega\) as
\[
|Df|(\Omega)\coloneqq\inf\bigg\{\varliminf_{n\to\infty}\int_\Omega\lip_a(f_n)\,\d\mm\;\bigg|\;
(f_n)_n\subseteq{\rm LIP}_{loc}(\Omega),\,f_n\to f|_\Omega\text{ in }L^1_{loc}(\Omega)\bigg\}.
\]
\end{definition}

If \(|Df|(\X)<\infty\), then there exists a unique finite outer-regular Borel measure \(|Df|\) on \(\X\)
that extends \(\Omega\mapsto|Df|(\Omega)\). We say that \(|Df|\) is the total variation measure of \(f\).
\medskip

Once BV functions have been introduced, we have a notion of a set of finite perimeter:
\begin{definition}[Set of finite perimeter]
Let \((\X,\sfd,\mm)\) be a metric measure space and \(E\subseteq\X\) a Borel set. Then we say that \(E\) is
a set of finite perimeter if \(|D\1_E|(\X)<\infty\). We call \(P(E)\coloneqq|D\1_E|(\X)\) the perimeter of \(E\)
and \(P(E,\cdot)\coloneqq|D\1_E|\) the perimeter measure of \(E\).
\end{definition}

We collect some properties of functions of bounded variation and sets of finite perimeter:
\begin{itemize}
\item If \(f,g\in L^1_{loc}(\X)\) with \(|Df|(\X),|Dg|(\X)<\infty\) are given, then \(|D(f+g)|\leq|Df|+|Dg|\).
\item It holds that \(|D(\lambda f)|=|\lambda||Df|\) for every \(\lambda\in\R\).
\item \textbf{Submodularity.} It holds that \(|D(f\vee g)|+|D(f\wedge g)|\leq|Df|+|Dg|\).
\item It holds that \(|Df|=|Df^+|+|Df^-|\). Indeed, on the one hand submodularity yields
\[
|Df^+|+|Df^-|=|D(f\vee 0)|+|D(-(f\wedge 0))|=|D(f\vee 0)|+|D(f\wedge 0)|\leq|Df|,
\]
while on the other hand \(f=f^+-f^-\) and thus \(|Df|=|D(f^+-f^-)|\leq|Df^+|+|Df^-|\).
\item \textbf{Leibniz rule.} If \(\eta\in{\rm LIP}_{bs}(\X)\), then \(|D(\eta f)|\leq|\eta||Df|+|f|\lip_a(\eta)\mm\).
\item \textbf{Lower semicontinuity.} If \(f_n\to f\) in \(L^1_{loc}(\X)\), then \(|Df|(\X)\leq\varliminf_n|Df_n|(\X)\).
\item If \(E\subseteq\X\) is a set of finite perimeter, then \(\X\setminus E\) is a set of finite perimeter and
\[
P(\X\setminus E,\cdot)=P(E,\cdot).
\]
\end{itemize}
\begin{theorem}[Fleming--Rishel coarea formula]
Let \((\X,\sfd,\mm)\) be a metric measure space and \(f\in L^1_{loc}(\X)\). Then
\(\R\ni t\mapsto P(\{f>t\})\in[0,\infty]\) is Borel measurable. Moreover, it holds that \(|Df|(\X)<\infty\)
if and only if \(\int_\R P(\{f>t\})\,\d t<+\infty\). In this case, for any Borel set \(E\subseteq\X\)
the function \(\R\ni t\mapsto P(\{f>t\},E)\in[0,\infty)\) is Borel measurable and it holds that
\[
|Df|(E)=\int_\R P(\{f>t\},E)\,\d t.
\]
\end{theorem}

We refer to \cite{Miranda} for proofs of the above results. We also record the following result,
which will be useful later, together with its proof for the reader's convenience.
\begin{proposition}[Truncations]\label{prop:TV_conv_truncated_functions}
Let \((\X,\sfd,\mm)\) be a metric measure space. Let \(f\in L^1_{loc}(\X)\) be such that \(|Df|(\X)<\infty\).
Define \(f_n\coloneqq(f\wedge n)\vee(-n)\) for every \(n\in\N\). Then it holds that
\[
|D(f-f_n)|(\X)\to 0\quad\text{ as }n\to\infty.
\]
\end{proposition}
\begin{proof}
Note that \(f-f_n=\1_{\{f>n\}}(f-n)+\1_{\{f<-n\}}(f+n)\). Then the coarea formula yields
\[\begin{split}
|D(f-f_n)|(\X)&\leq|D(\1_{\{f>n\}}(f-n))|(\X)+|D(\1_{\{f<-n\}}(f+n))|(\X)\\
&=\int_\R P(\{\1_{\{f>n\}}(f-n)>t\})\,\d t+\int_\R P(\{\1_{\{f<-n\}}(f+n)>t\})\,\d t\\
&=\int_0^\infty P(\{f>n+t\})\,\d t+\int_{-\infty}^0 P(\{f>-n+t\})\,\d t\\
&=\int_{\R\setminus[-n,n]}P(\{f>t\})\,\d t.
\end{split}\]
Since \(\int_\R P(\{f>t\})\,\d t=|Df|(\X)<\infty\), by applying the dominated convergence theorem
we conclude that \(|D(f-f_n)|(\X)\leq\int_{\R\setminus[-n,n]}P(\{f>t\})\,\d t\to 0\) as \(n\to\infty\).
\end{proof}
\subsection{BV functions via derivations}
Given a metric measure space \((\X,\sfd,\mm)\), we denote
\[
{\rm BV}(\X)\coloneqq\{f\in L^1(\X):|Df|(\X)<\infty\}.
\]
Di Marino obtained in \cite{DiMar:14} an equivalent definition -- in terms of a suitable notion of Lipschitz derivation
-- of the space \({\rm BV}(\X)\) and of the total variation measures of its elements. We denote by
\({\rm Der}^\infty_\infty(\X)\) the collection of all linear operators \(b\colon{\rm LIP}_{bs}(\X)\to L^\infty(\X)\)
such that:
\begin{itemize}
\item \textbf{Leibniz rule.} It holds that \(b(fg)=g\,b(f)+f\,b(g)\) for every \(f,g\in{\rm LIP}_{bs}(\X)\).
\item \textbf{Weak locality.} There exists a non-negative function \(G\in L^\infty(\X)\) such that
\[
|b(f)|\leq G\,\lip_a(f)\quad\text{ for every }f\in{\rm LIP}_{bs}(\X).
\]
The \(\mm\)-a.e.\ minimal such function \(G\) is denoted by \(|b|\).
\item \textbf{Divergence.} There exists a (unique) function \({\rm div}(b)\in L^\infty(\X)\) such that
\[
\int b(f)\,\d\mm=-\int f\,{\rm div}(b)\,\d\mm\quad\text{ for every }f\in{\rm LIP}_{bs}(\X).
\]
\end{itemize}
Given any \(b\in{\rm Der}^\infty_\infty(\X)\) and \(h\in L^\infty(\X)\), we define
\(hb\colon{\rm LIP}_{bs}(\X)\to L^\infty(\X)\) as
\[
(hb)(f)\coloneqq h\,b(f)\quad\text{ for every }f\in{\rm LIP}_{bs}(\X).
\]
For any \(h\in{\rm LIP}_{bs}(\X)\), we have \(hb\in{\rm Der}^\infty_\infty(\X)\), \(|hb|=|h||b|\) and \({\rm div}(hb)=h\,{\rm div}(b)+b(h)\).
\begin{remark}{\rm
In the Euclidean space \(\R^n\), endowed with the Euclidean distance and the Lebesgue measure, the elements of
\({\rm Der}^\infty_\infty(\R^n)\) can be identified with those \(L^\infty\)-vector fields \(v\colon\R^n\to\R^n\)
whose distributional divergence is in \(L^\infty(\R^n)\); see \cite[Theorem 5.6]{LucicGelli}.
\fr}\end{remark}

In the following result, whose validity is a consequence of \cite[Theorems 3.7 and 3.3]{DiMar:14},
we denote by \(\mathfrak M(\X)\) the Banach space of all finite signed Radon measures on \(\X\).
\begin{theorem}
Let \((\X,\sfd,\mm)\) be a metric measure space and \(f\in L^1(\X)\). Then \(f\in{\rm BV}(\X)\) if and only if there
exists a (unique) bounded linear operator \({\bf L}_f\colon{\rm Der}^\infty_\infty(\X)\to\mathfrak M(\X)\) such that:
\begin{itemize}
\item[\(\rm i)\)] \({\bf L}_f(hb)=h\,{\bf L}_f(b)\) for every \(h\in{\rm LIP}_{bs}(\X)\) and
\(b\in{\rm Der}^\infty_\infty(\X)\).
\item[\(\rm ii)\)] \({\bf L}_f(b)(\X)=-\int f\,{\rm div}(b)\,\d\mm\) for every \(b\in{\rm Der}^\infty_\infty(\X)\).
\end{itemize}
\end{theorem}
\begin{remark}{\rm
Fix any \(b\in{\rm Der}^\infty_\infty(\X)\) and an upper semicontinuous function
\(h\colon\X\to[0,\infty)\) such that \(|b|\leq h\) holds \(\mm\)-a.e.\ in \(\X\). Then we have that
\begin{equation}\label{eq:ptwse_ineq_TV_BV_der}
|{\bf L}_f(b)|\leq h|Df|\quad\text{ for every }f\in{\rm BV}(\X).
\end{equation}
The validity of the above claim follows from \cite[Theorem 3.3]{DiMar:14}.
\fr}\end{remark}

For any \(b\in{\rm Der}^\infty_\infty(\X)\), we set \({\rm spt}(b)\coloneqq{\rm spt}(|b|)\). Given
\(B\subseteq\Omega\subseteq\X\) with \(\Omega\) open, we write \(B\Subset\Omega\) if \(B\) is bounded
and \(\sfd(\X\setminus\Omega,B)>0\). The next result is taken from \cite[Theorem 3.4]{DiMar:14}:
\begin{theorem}\label{thm:repr_per}
Let \((\X,\sfd,\mm)\) be a metric measure space and \(f\in{\rm BV}(\X)\). Let \(\Omega\subseteq\X\)
be an open set. Then it holds that
\[
|Df|(\Omega)=\sup\bigg\{\int_\Omega f\,{\rm div}(b)\,\d\mm\;\bigg|\;b\in{\rm Der}^\infty_\infty(\X),\,
|b|\leq 1,\,{\rm spt}(b)\Subset\Omega\bigg\}.
\]
\end{theorem}
\subsection{Functions of bounded variation in PI spaces}
We focus our attention on a special class of metric measure spaces:
\begin{definition}[PI space]
We say that a metric measure space \((\X,\sfd,\mm)\) is a PI space provided it satisfies the following two conditions:
\begin{itemize}
\item[\(\rm i)\)] \textbf{Doubling condition.} There exists a constant \(C_D>0\) such that
\[
0<\mm(B(x,2r))\leq C_D\mm(B(x,r))\quad\text{ for every }x\in\X\text{ and }r>0.
\]
\item[\(\rm ii)\)] \textbf{Weak \(1\)-Poincar\'{e} inequality.} There exist constants \(C_P>0\) and \(\lambda\geq 1\) such that
\[
\fint_{B(x,r)}\bigg|f-\fint_{B(x,r)}f\,\d\mm\bigg|\,\d\mm\leq C_P r\fint_{B(x,\lambda r)}\lip_a(f)\,\d\mm
\]
for every \(f\in{\rm LIP}_{loc}(\X)\), \(x\in\X\) and \(r>0\).
\end{itemize}
\end{definition}

Furthermore, a metric measure space \((\X,\sfd,\mm)\) is said to be Ahlfors \(q\)-regular, for some exponent \(q\geq 1\),
provided there exists a constant \(C_A\geq 1\) such that
\[
\frac{1}{C_A}r^q\leq\mm(B(x,r))\leq C_A r^q\quad\text{ for every }x\in\X\text{ and }r\in(0,{\rm diam}(\X)).
\]
The above condition implies that \((\X,\sfd,\mm)\) is doubling with \(C_D=C_A^2 2^q\).
\medskip

The following result was proved in \cite[Theorem 4.5]{Miranda}.
\begin{theorem}[Isoperimetric inequality]
Let \((\X,\sfd,\mm)\) be an Ahlfors \(q\)-regular PI space with \(q>1\). Then there exists a constant \(C_I>0\),
depending only on \(q\) and \(C_P\), such that
\begin{equation}\label{eq:isoper_ineq}
\min\{\mm(E),\mm(\X\setminus E)\}\leq C_I P(E)^{q/(q-1)}\quad\text{ for every set of finite perimeter }E\subseteq\X.
\end{equation}
\end{theorem}

The following result is well known to experts; however, in the absence of a suitable reference,
we include a detailed proof.
\begin{theorem}[Embedding theorem]\label{thm:BV_embed}
Let \((\X,\sfd,\mm)\) be an Ahlfors \(q\)-regular PI space with \(q>1\) and \(\mm(\X)=\infty\).
Let \(f\in L^1_{loc}(\X)\) be such that \(|Df|(\X)<\infty\). Then there exists a unique
number \(c_f\in\R\)
such that \(f-c_f\in L^{q/(q-1)}(\X)\), and it holds that
\begin{equation}\label{eq:BV_embed_ineq}
\|f-c_f\|_{L^{q/(q-1)}(\X)}\leq 2C_I|Df|(\X).
\end{equation}
Moreover, if in addition \(f\geq 0\) and \(\mm(\{f>t\})<\infty\) for a.e.\ \(t>0\), then \(c_f=0\).
\end{theorem}
\begin{proof}
Since \(\{f>t\}\) is a set of finite perimeter for a.e.\ \(t\in\R\) by the coarea formula, it follows from \eqref{eq:isoper_ineq}
that for a.e.\ \(t\in\R\) we have either \(\mm(\{f>t\})<\infty\) or \(\mm(\{f\leq t\})<\infty\).
Up to replacing \(f\) with \(-f\), we can assume that \(\mm(\{f>t_0\})<\infty\) for some \(t_0\in\R\).
Since \(\R\ni t\mapsto\mm(\{f>t\})\in[0,\infty]\) is non-increasing, there exists \(c_f\in[-\infty,\infty)\) such that
\begin{equation}\label{eq:BV_embed_aux}\begin{split}
\mm(\{f\leq t\})<\infty=\mm(\{f>t\})&\quad\text{ for every }t\in(-\infty,c_f),\\
\mm(\{f>t\})<\infty=\mm(\{f\leq t\})&\quad\text{ for every }t\in(c_f,\infty).
\end{split}\end{equation}
We claim that \(c_f\neq-\infty\). To prove it, we argue by contradiction: assume that \(c_f=-\infty\),
so that \(\mm(\{f>t\})<\infty\) for every \(t\in\R\). Then \eqref{eq:isoper_ineq} and the coarea formula imply that
\[\begin{split}
\infty&=\mm(\X)^{(q-1)/q}=\lim_{\lambda\to-\infty}\mm(\{f>\lambda\})^{(q-1)/q}\leq
\varliminf_{\lambda\to-\infty}\int_{\lambda-1}^\lambda\mm(\{f>t\})^{(q-1)/q}\,\d t\\
&\leq C_I^{(q-1)/q}\varliminf_{\lambda\to-\infty}\int_{\lambda-1}^\lambda P(\{f>t\})\,\d t
\leq C_I^{(q-1)/q}|Df|(\X)<\infty,
\end{split}\]
thus leading to a contradiction. All in all, we have shown that \(c_f\in\R\).

Denote \(f_\pm\coloneqq(f-c_f)^\pm\). Applying the layer cake representation formula, \cite[Lemma 3.48]{AFP},
the isoperimetric inequality, \eqref{eq:BV_embed_aux} and the coarea formula, we obtain that
\[\begin{split}
\int f_\pm^{q/(q-1)}\,\d\mm&=\int_0^\infty\mm(\{f_\pm^{q/(q-1)}>t\})\,\d t=\frac{q}{q-1}\int_0^\infty\mm(\{f_\pm>t\})t^{1/(q-1)}\,\d t\\
&\leq\bigg(\int_0^\infty\mm(\{f_\pm>t\})^{(q-1)/q}\,\d t\bigg)^{q/(q-1)}\\
&\leq C_I^{q/(q-1)}\bigg(\int_0^\infty P(\{f_\pm>t\})\,\d t\bigg)^{q/(q-1)}\\
&=C_I^{q/(q-1)}|Df_\pm|(\X)^{q/(q-1)}\leq C_I^{q/(q-1)}|Df|(\X)^{q/(q-1)},
\end{split}\]
whence it follows that \(\|f-c_f\|_{L^{q/(q-1)}(\X)}\leq\|f_+\|_{L^{q/(q-1)}(\X)}+\|f_-\|_{L^{q/(q-1)}(\X)}\leq 2C_I|Df|(\X)\).
Since \(\mm(\X)=\infty\), we have that \(c_f\) is the unique real number satisfying \(f-c_f\in L^{q/(q-1)}(\X)\). Finally, if in addition \(f\geq 0\) and \(\mm(\{f>t\})<\infty\) for a.e.\ \(t>0\), then \eqref{eq:BV_embed_aux}
implies \(c_f=0\).
\end{proof}

Let \((\X,\sfd,\mm)\) be an Ahlfors \(q\)-regular PI space with \(q>1\) and \(\mm(\X)=\infty\). Then we define
\begin{equation}\label{eq:def_FV}
{\rm FV}(\X)\coloneqq\big\{f\in L^{q/(q-1)}(\X)\;\big|\;|Df|(\X)<\infty\big\}.
\end{equation}
Note that there is no need to specify the exponent \(q\) in the notation \(\FV(\X)\), since it is the unique value
for which \(\X\) is Ahlfors \(q\)-regular.
We also define \(\|f\|_{{\rm FV}(\X)}\coloneqq|Df|(\X)\) for every \(f\in{\rm FV}(\X)\). Theorem
\ref{thm:BV_embed} implies that
\begin{equation}\label{eq:equiv_norm_FV}
\|f\|_{{\rm FV}(\X)}\leq\|f\|_{L^{q/(q-1)}(\X)}+|Df|(\X)\leq(2C_I+1)\|f\|_{{\rm FV}(\X)}
\quad\text{ for every }f\in{\rm FV}(\X),
\end{equation}
whence it easily follows that \(({\rm FV}(\X),\|\cdot\|_{{\rm FV}(\X)})\) is a Banach space.
\subsection{Indecomposable sets and monotone functions}
We now recall the definition of an indecomposable set, first considered in the metric setting in \cite{BPR}.
\begin{definition}[Indecomposable set]
Let \((\X,\sfd,\mm)\) be a metric measure space. Then we say that a set of finite perimeter \(E\subseteq\X\)
is decomposable if there exist two disjoint sets of finite perimeter \(G,H\subseteq\X\) with \(\mm(G),\mm(H)>0\) such
that \(E=G\cup H\) and \(P(E)=P(G)+P(H)\). If \(E\) is not decomposable, then we say that it is indecomposable.
\end{definition}

The following decomposition result was established in \cite{BPR,L-decom}.
\begin{theorem}[Decomposition into indecomposable components]
Let \((\X,\sfd,\mm)\) be a PI space and \(E\subseteq\X\) a set of finite perimeter. Then there exists an
\(\mm\)-a.e.\ unique at most countable partition \(\{E_i:i\in I\}\) of \(E\) into maximal indecomposable
sets, which we call
the essential connected components of \(E\), such that \(\mm(E_i)>0\) for every \(i\in I\) and
\[
P(E,\cdot)=\sum_{i\in I}P(E_i,\cdot).
\]
The maximality of \(\{E_i:i\in I\}\) is intended in the following sense: given any indecomposable set \(F\subseteq E\),
there exists a (unique) index \(i_F\in I\) such that \(F\subseteq_\mm E_{i_F}\).
\end{theorem}

Let us recall some other notions and results that are related to indecomposability. Assume \((\X,\sfd,\mm)\)
is a PI space with \(\mm(\X)=\infty\), and \(E\subseteq\X\) is an indecomposable set. Then:
\begin{itemize}
\item We call holes of \(E\) those essential connected components of \(\X\setminus E\) of finite \(\mm\)-measure.
\item The saturation \({\rm sat}(E)\) of \(E\) is defined as the union of \(E\) and all its holes.
\item The saturation \({\rm sat}(E)\) is indecomposable \cite[Proposition 4.13(ii)]{L-decom}.
\item Assuming in addition that \((\X,\sfd,\mm)\) is Ahlfors \(q\)-regular for \(q>1\) and \(\mm(E)<\infty\),
we have that \(\X\setminus E\) has exactly one essential connected component \(G\) with \(\mm(G)=\infty\); see
\cite[Proposition 3.14]{BPR} (in light of \cite{L-decom}, the isotropicity assumption can be dropped).
Hence, we have \(G=\X\setminus{\rm sat}(E)\), thus \(\mm({\rm sat}(E))<\infty\) by the isoperimetric inequality.
Moreover, if \(F\subseteq\X\) is an indecomposable set with \(F\subseteq_\mm E\), then \({\rm sat}(F)\subseteq_\mm{\rm sat}(E)\).
\end{itemize}
\begin{definition}[Monotone function of bounded variation]
Let \((\X,\sfd,\mm)\) be a metric measure space. Let \(f\in L^1_{loc}(\X)\) be such that \(|Df|(\X)<\infty\).
Then we say that \(f\) is monotone provided the sets \(\{f>t\}\) and \(\{f\leq t\}\) are indecomposable
for a.e.\ \(t\in\R\).
\end{definition}
\section{Decomposition of BV functions}
\subsection{Superlevel sets}
Let \((\X,\sfd,\mm)\) be a metric measure space. Given Borel sets \(A,B\subseteq\X\), we write
\(A=_\mm B\) to indicate that \(\mm(A\Delta B)=0\). Likewise, we write \(A\subseteq_\mm B\) if
\(A\setminus B=_\mm\varnothing\).
\medskip

In the sequel, we shall construct BV functions by declaring the collection of their superlevel sets.
To this purpose, we give the following auxiliary definition:
\begin{definition}[The class \(\mathscr D(\X)\)]
Let \((\X,\sfd,\mm)\) be a metric measure space. Then we denote by \(\mathscr D(\X)\) the collection of all
couples \(A=(\Delta_A,\{A_t:t\in\Delta_A\})\) such that:
\begin{itemize}
\item[\(\rm i)\)] \(\Delta_A\subseteq(0,\infty)\) is a Borel set such that \(\mathscr L^1((0,\infty)\setminus\Delta_A)=0\).
\item[\(\rm ii)\)] Each \(A_t\), \(t\in\Delta_A\), is a Borel subset of \(\X\) and \(A_s\subseteq A_t\) for every
\(t,s\in\Delta_A\) with \(t<s\).
\item[\(\rm iii)\)] For any \(x\in\X\) there exists \(r_x>0\) such that \(\int_0^\infty\mm(A_t\cap B(x,r_x))\,\d t<\infty\).
\end{itemize}
\end{definition}

Item iii) is well posed, because \(\Delta_A\ni t\mapsto\mm(A_t\cap B(x,r_x))\) is non-increasing and thus Borel.
Clearly, if \(f\colon\X\to[0,+\infty)\) is locally integrable, then
\(((0,\infty),\{\{f>t\}:t>0\})\in\mathscr D(\X)\). In the next result we prove the converse,
namely that each element \(A\) of \(\mathscr D(\X)\) induces a function \(\phi_A\) whose superlevel
sets are those prescribed by \(A\).
\begin{proposition}\label{prop:from_A_to_phi_A}
Let \((\X,\sfd,\mm)\) be a metric measure space and \(A\in\mathscr D(\X)\). Then there exists a Borel function
\(\phi_A\colon\X\to[0,\infty)\) such that
\begin{equation}\label{eq:from_A_to_phi_A}
\{\phi_A>t\}=_\mm A_t\quad\text{ for all but countably many }t\in\Delta_A.
\end{equation}
Moreover, the function \(\phi_A\) is locally integrable and \(\mm\)-a.e.\ unique.
\end{proposition}
\begin{proof}
Recursively, we take strictly increasing sequences \((t^n_i)_{i\in\N}\subseteq\Delta_A\), for \(n\in\N\), such that:
\begin{itemize}
\item \(\{t^n_i:i\in\N\}\subseteq\{t^{n+1}_i:i\in\N\}\) for every \(n\in\N\),
\item \(|t^n_i-t^n_{i+1}|<2^{-n}\) for every \(i,n\in\N\),
\item \(\lim_{n\to\infty}t^n_1=0\),
\item \(\lim_{i\to\infty}t^n_i=\infty\) for every \(n\in\N\).
\end{itemize}
We then define the functions \(\phi_A^n\colon\X\to[0,+\infty)\) as
\[
\phi_A^n\coloneqq\sum_{i=1}^\infty t^n_i\1_{A_{t^n_i}\setminus A_{t^n_{i+1}}}.
\]
Since \((\phi_A^n(x))_n\) is non-decreasing for every \(x\in\X\), it makes sense to define \(\phi_A(x)\coloneqq\lim_n\phi_A^n(x)\) for every \(x\in\X\).
Clearly, \(\phi_A\) is Borel measurable by construction. Next, let us define
\[
A_t^+\coloneqq\bigcup_{s\in\Delta_A\cap(t,+\infty)}A_s,\qquad A_t^-\coloneqq\bigcap_{s\in\Delta_A\cap(0,t)}A_s
\quad\text{ for every }t\in\Delta_A.
\]
Each set \(A_t^\pm\) is Borel, as it can be expressed as a countable union/intersection of \(A_s\)'s. Since the sets
\(\{A_t^-\setminus A_t^+:t\in\Delta_A\}\) are pairwise disjoint and \(\mm\) is \(\sigma\)-finite, we have \(\mm(A_t^-\setminus A_t^+)=0\),
thus \(A_t=_\mm A_t^-=_\mm A_t^+\), and \(\{\phi_A=t\}=_\mm\varnothing\)
for all but countably many \(t\in\Delta_A\). Fix any such \(t\) that in addition does not belong to \(\{t^n_i:i,n\in\N\}\).
Take \(n_t\in\N\) such that \(t^n_1<t\) for every \(n\geq n_t\). Then for any \(n\geq n_t\) there exists a unique index \(i_n\in\N\)
such that \(t^n_{i_n}<t<t^n_{i_n+1}\). Note that \(\lim_n t^n_{i_n}=\lim_n t^n_{i_n+1}=t\). On the one hand, we have
that \(A_{t^n_i}\subseteq\{\phi_A^n\geq t^n_i\}\subseteq\{\phi_A\geq t^n_i\}\) for every \(i,n\in\N\), thus accordingly
\begin{equation}\label{eq:from_A_to_phi_A_aux}
A_t^-=\bigcap_{n\geq n_t}A_{t^n_{i_n}}\subseteq\bigcap_{n\geq n_t}\{\phi_A\geq t^n_{i_n}\}=\{\phi_A\geq t\}.
\end{equation}
On the other hand, if \(x\in\{\phi_A>t\}\) is given, then we can find \(n(x)\in\N\) with \(n(x)\geq n_t\) such that
\(\phi_A^{n(x)}(x)>t\), thus \(x\in A_{t^{n(x)}_{i_{n(x)}+1}}\). This fact, together with \eqref{eq:from_A_to_phi_A_aux},
implies that
\[
\{\phi_A>t\}\subseteq\bigcup_{n\geq n_t}A_{t^n_{i_n+1}}=A_t^+\subseteq A_t^-\subseteq\{\phi_A\geq t\}.
\]
Since \(\{\phi_A\geq t\}\setminus\{\phi_A>t\}=\{\phi_A=t\}=_\mm\varnothing\)
thanks to our choice of \(t\), we finally obtain \eqref{eq:from_A_to_phi_A}.

To conclude, it remains to show \(\mm\)-a.e.\ uniqueness. Suppose \(\tilde\phi_A\) is another function verifying the claim.
Then \(h\coloneqq\min\{\phi_A,\tilde\phi_A\}\) verifies the claim, as \(\{h>t\}=\{\phi_A>t\}\cap\{\tilde\phi_A>t\}=_\mm A_t\)
for all but countably many \(t\in\Delta_A\). Since \((\X,\sfd)\) is separable, we can find a sequence \((B_k)_{k\in\N}\)
of open balls in \(\X\) such that \(\X=\bigcup_{k\in\N}B_k\) and \(\int_0^\infty\mm(A_t\cap B_k)\,\d t<+\infty\)
for every \(k\in\N\). Therefore, the layer cake representation formula yields
\[\begin{split}
\int_{B_k}h\,\d\mm&=\int_0^\infty\mm(\{h>t\}\cap B_k)\,\d t=\int_0^\infty\mm(A_t\cap B_k)\,\d t=\int_0^\infty\mm(\{\phi_A>t\}\cap B_k)\,\d t\\
&=\int_{B_k}\phi_A\,\d\mm,
\end{split}\]
and similarly \(\int_{B_k}h\,\d\mm=\int_{B_k}\tilde\phi_A\,\d\mm\), for every \(k\in\N\). This shows that \(\phi_A=h=\tilde\phi_A\) holds \(\mm\)-a.e.\ on \(\X\)
and that \(\phi_A\) is locally integrable.
\end{proof}

For any \(A\in\mathscr D(\X)\), it follows from the coarea formula that
\[
|D\phi_A|(\X)<\infty\quad\Longleftrightarrow\quad\int_0^\infty P(A_t)\,\d t<\infty.
\]
In particular, if \(|D\phi_A|(\X)<\infty\), then \(A_t\) is a set of finite perimeter for a.e.\ \(t\in\Delta_A\).
\subsection{Construction of the decomposition}
We now turn to the auxiliary results needed to prove our main result, Theorem \ref{thm:main}.
\begin{lemma}\label{lem:additiv_TV}
Let \((\X,\sfd,\mm)\) be a metric measure space. Let \(f,g\in L^1_{loc}(\X)\) be such that \(0\leq f\leq g\)
and \(|Df|(\X),|Dg|(\X)<\infty\). Assume that there exists \(\sigma\in\{-1,1\}\) such that
\begin{equation}\label{eq:additiv_TV}
P(\{g>t\})=P(\{f>t\})+\sigma P(\{g>t\}\setminus\{f>t\})\quad\text{ for a.e.\ }t>0.
\end{equation}
Then it holds that \(|Dg|=|Df|+\sigma|D(g-f)|\).
\end{lemma}
\begin{proof}
Define \(g_n\coloneqq(g\wedge n)\vee(-n)\) and \(f_n\coloneqq(f\wedge n)\vee(-n)\) for every \(n\in\N\). Note that
\[
\{g_n>t\}=\left\{\begin{array}{lll}
\X\\
\{g>t\}\\
\varnothing
\end{array}\quad\begin{array}{lll}
\text{ for every }t\in(-\infty,-n),\\
\text{ for every }t\in[-n,n),\\
\text{ for every }t\in[n,\infty),
\end{array}\right.
\]
and similarly for \(\{f_n>t\}\). Accordingly, \eqref{eq:additiv_TV} implies that for any \(n\in\N\) it holds that
\begin{equation}\label{eq:additiv_TV_aux}
P(\{g_n>t\})=P(\{f_n>t\})+\sigma P(\{g_n>t\}\setminus\{f_n>t\})\quad\text{ for a.e.\ }t>0.
\end{equation}
Now, fix a point \(\bar x\in\X\), and denote by \(B_k\) the open ball of center \(\bar x\) and radius \(k\in\N\).
For any \(k\in\N\), take a boundedly-supported Lipschitz function \(\eta_k\colon\X\to[0,1]\) with \(\eta_k=1\)
in \(B_k\). Note that \(\eta_k(g_n-f_n)\in{\rm BV}(\X)\) by the Leibniz rule. Likewise,
\(\eta_k\1_{\{g_n>t\}\setminus\{f_n>t\}}\in{\rm BV}(\X)\) for a.e.\ \(t>0\) by the coarea formula and
the Leibniz rule. Now, fix a derivation \(b\in{\rm Der}^\infty_\infty(\X)\) with \(|b|\leq 1\).
Take a sequence of functions \((h^k_j)_{j\in\N}\subseteq{\rm LIP}(\X;[0,1])\)
such that \(h^k_j(x)\nearrow\1_{B_k}(x)\) as \(j\to\infty\) for every \(x\in\X\)
and \({\rm spt}(h^k_j)\subseteq B_k\) for every \(j\in\N\).
Note that \(h^k_j b\) belongs to \({\rm Der}^\infty_\infty(\X)\) and \(|h^k_j b|\leq\1_{{\rm spt}(h^k_j)}\).
The layer cake representation formula ensures that
\[
g_n(x)-f_n(x)=\int_0^\infty\1_{\{g_n>t\}}(x)\,\d t-\int_0^\infty\1_{\{f_n>t\}}(x)\,\d t
=\int_0^\infty\1_{\{g_n>t\}\setminus\{f_n>t\}}(x)\,\d t
\]
for \(\mm\)-a.e.\ \(x\in\X\). Taking also Fubini's theorem, the coarea formula, \eqref{eq:ptwse_ineq_TV_BV_der}, the Leibniz rule
and \eqref{eq:additiv_TV_aux} into account, we deduce that
\[\begin{split}
&\big|{\bf L}_{\eta_k(g_n-f_n)}(h^k_j b)(\X)\big|=\bigg|\int\eta_k(g_n-f_n)\,{\rm div}(h^k_j b)\,\d\mm\bigg|\\
=\,&\bigg|\int_0^\infty\!\!\!\int\eta_k\1_{\{g_n>t\}\setminus\{f_n>t\}}\,{\rm div}(h^k_j b)\,\d\mm\,\d t\bigg|
=\bigg|\int_0^\infty{\bf L}_{\eta_k\1_{\{g_n>t\}\setminus\{f_n>t\}}}(h^k_j b)(\X)\,\d t\bigg|\\
\leq\,&\int_0^\infty|D(\eta_k\1_{\{g_n>t\}\setminus\{f_n>t\}})|({\rm spt}(h^k_j))\,\d t
\leq\int_0^\infty|D(\eta_k\1_{\{g_n>t\}\setminus\{f_n>t\}})|(B_k)\,\d t\\
\leq\;&\int_0^\infty\!\!\!\int_{B_k}\eta_k\,\d P(\{g_n>t\}\setminus\{f_n>t\},\cdot)\,\d t
+\int_0^\infty\!\!\!\int_{B_k}\1_{\{g_n>t\}\setminus\{f_n>t\}}\,\lip_a(\eta_k)\,\d\mm\,\d t\\
\leq\,&\int_0^\infty P(\{g_n>t\}\setminus\{f_n>t\})\,\d t=\sigma\int_0^\infty P(\{g_n>t\})-P(\{f_n>t\})\,\d t\\
=\,&\sigma|Dg_n|(\X)-\sigma|Df_n|(\X).
\end{split}\]
By applying the monotone convergence theorem, we obtain that
\[\begin{split}
\big|{\bf L}_{\eta_k(g_n-f_n)}(b)(B_k)-{\bf L}_{\eta_k(g_n-f_n)}(h^k_j b)(\X)\big|
&=\bigg|\int\1_{B_k}-h^k_j\,\d{\bf L}_{\eta_k(g_n-f_n)}(b)\bigg|\\
&\leq\int\1_{B_k}-h^k_j\,\d|{\bf L}_{\eta_k(g_n-f_n)}(b)|\to 0
\end{split}\]
as \(j\to\infty\). Therefore, given any \(n,k\in\N\), it follows from the above estimates that
\begin{equation}\label{eq:additiv_TV_aux2}
|{\bf L}_{\eta_k(g_n-f_n)}(b)(B_k)|\leq\sigma|Dg_n|(\X)-\sigma|Df_n|(\X)
\quad\text{ for every }b\in{\rm Der}^\infty_\infty(\X)\text{ with }|b|\leq 1.
\end{equation}
By Theorem \ref{thm:repr_per}, we can find \(b_{n,k}\in{\rm Der}^\infty_\infty(\X)\) with \(|b_{n,k}|\leq 1\)
such that \({\rm spt}(b_{n,k})\Subset B_k\) and
\begin{equation}\label{eq:additiv_TV_aux3}
|{\bf L}_{\eta_k(g_n-f_n)}(b_{n,k})(\X)|\geq\int_{B_k}\eta_k(g_n-f_n)\,{\rm div}(b_{n,k})\,\d\mm
\geq|D(\eta_k(g_n-f_n))|(B_k)-\frac{1}{k}.
\end{equation}
Moreover, we have that \({\bf L}_{\eta_k(g_n-f_n)}(b_{n,k})(\X)={\bf L}_{\eta_k(g_n-f_n)}(b_{n,k})(B_k)\),
because \eqref{eq:ptwse_ineq_TV_BV_der} implies that
\[
\big|{\bf L}_{\eta_k(g_n-f_n)}(b_{n,k})(\X\setminus B_k)\big|
\leq|D(\eta_k(g_n-f_n))|({\rm spt}(b_{n,k})\setminus B_k)=0.
\]
Combining \eqref{eq:additiv_TV_aux2} with \eqref{eq:additiv_TV_aux3}, we thus obtain that
\[
|D(g_n-f_n)|(B_k)=|D(\eta_k(g_n-f_n))|(B_k)\leq\sigma|Dg_n|(\X)-\sigma|Df_n|(\X)+\frac{1}{k}
\quad\text{ for all }n,k\in\N.
\]
Letting \(k\to\infty\), we deduce that \(|D(g_n-f_n)|(\X)\leq\sigma|Dg_n|(\X)-\sigma|Df_n|(\X)\)
for every \(n\in\N\). By Proposition \ref{prop:TV_conv_truncated_functions}, letting \(n\to\infty\)
we get that \(|D(g-f)|(\X)\leq\sigma|Dg|(\X)-\sigma|Df|(\X)\). Since
\[
\sigma|Dg|-\sigma|Df|=\left\{\begin{array}{ll}
|Df|-|Dg|\leq|D(f-g)|=|D(g-f)|\\
|Dg|-|Df|\leq|D(g-f)|
\end{array}\quad\begin{array}{ll}
\text{ if }\sigma=-1,\\
\text{ if }\sigma=1,
\end{array}\right.
\]
we finally conclude that \(|D(g-f)|=\sigma|Dg|-\sigma|Df|\), whence the statement follows.
\end{proof}

Given a Borel function \(f\colon\X\to[0,\infty)\) such that \(|Df|(\X)<\infty\), we denote
by \(D_f\) the set of all \(t>0\) such that \(\{f>t\}\) is a set of finite perimeter. Recall
that \(\mathscr L^1([0,\infty)\setminus D_f)=0\).
\begin{proposition}\label{prop:extract_monot}
Let \((\X,\sfd,\mm)\) be an Ahlfors \(q\)-regular PI space with \(q>1\) and \(\mm(\X)=\infty\).
Fix any \(f\in{\rm FV}(\X)\setminus\{0\}\). Then there exists a monotone function
\(g\in{\rm FV}(\X)\setminus\{0\}\) such that
\[
|Df|=|D(f-g)|+|Dg|.
\]
\end{proposition}
\begin{proof}
Up to replacing \(f\) with \(-f\), we can assume that \(\mm(\{f>0\})>0\), thus \(\mm(\{f>t_0\})>0\) holds
for some \(t_0\in D_f\). Fix some essential connected component \(C_{t_0}\) of \(\{f>t_0\}\). Given any
\(t\in D_f\cap(0,t_0)\), we have that \(\{f>t_0\}\subseteq_\mm\{f>t\}\), thus there exists
a unique essential connected component \(C_t\) of \(\{f>t\}\) such that \(C_{t_0}\subseteq_\mm C_t\). Moreover,
define \(S_t\coloneqq{\rm sat}(C_t)\) for every \(t\in D_f\cap(0,t_0]\) and \(C_t=S_t\coloneqq\varnothing\) for
every \(t\in D_f\cap(t_0,\infty)\). Without loss of generality, we can assume that \(C_s\subseteq C_t\)
and \(S_s\subseteq S_t\) for every \(t,s\in D_f\) with \(t<s\). Observe that
\[\begin{split}
\int_0^{t_0}\mm(C_t\cap B_1(x))\,\d t&\leq\int_0^\infty\mm(\{f>t\}\cap B_1(x))\,\d t=\int_{B_1(x)}f^+\,\d\mm\\
&\leq\mm(B_1(x))^{1/q}\,\|f\|_{L^{q/(q-1)}(\X)}<\infty\quad\text{ for every }x\in\X,
\end{split}\]
thus \(C\coloneqq(D_f,\{C_t:t\in D_f\})\in\mathscr D(\X)\).
As \(f\in L^{q/(q-1)}(\X)\) implies
\(\mm(C_t)\leq\mm(\{f>t\})<\infty\) for a.e.\ \(t\in D_f\),
we have that \(\phi_C\in L^{q/(q-1)}(\X)\)
and \(\mm(S_t)<\infty\) for a.e.\ \(t\in D_f\). In particular,
\[\begin{split}
\int_0^{t_0}\mm(S_t\cap B_1(x))\,\d t&\leq\mm(B_1(x))^{1/q}\int_0^{t_0}\mm(S_t)^{(q-1)/q}\,\d t
\leq C_I^{(q-1)/q}\mm(B_1(x))^{1/q}\int_0^{t_0}P(S_t)\,\d t\\
&\leq C_I^{(q-1)/q}\mm(B_1(x))^{1/q}\int_0^\infty P(\{f>t\})\,\d t<\infty\quad\text{ for every }x\in\X
\end{split}\]
by H\"{o}lder's inequality and the isoperimetric inequality.
Then \(S\coloneqq(D_f,\{S_t:t\in D_f\})\in\mathscr D(\X)\),
thus \(g\coloneqq\phi_S\in L^{q/(q-1)}(\X)\) by Theorem \ref{thm:BV_embed}.
Since \(\phi_C\leq f^+\) and \(P(\{f>t\})=P(\{f>t\}\setminus C_t)+P(C_t)\) for every \(t\in D_f\),
Lemma \ref{lem:additiv_TV} yields \(|Df^+|=|D(f^+-\phi_C)|+|D\phi_C|\).
Given that \(\phi_C\leq g\) and \(P(C_t)=P(S_t\setminus C_t)+P(S_t)\) for every \(t\in D_f\),
Lemma \ref{lem:additiv_TV} yields \(|D\phi_C|=|D(\phi_C-g)|+|Dg|\). Recalling that
\(|Df|=|Df^+|+|Df^-|\), we get
\[\begin{split}
|Df|&=|D(f^+-\phi_C)|+|D(\phi_C-g)|+|Dg|+|Df^-|\\
&\geq|D(f^+-g)|+|D(-f^-)|+|Dg|\geq|D(f-g)|+|Dg|\geq|Df|,
\end{split}\]
whence \(|Df|=|D(f-g)|+|Dg|\) follows. As \(g\neq 0\) by construction, the proof is complete.
\end{proof}

We assume the validity of the axiom \({\sf DC}_{\omega_1}\), i.e.\ the Dependent Choice indexed over
the first uncountable ordinal \(\omega_1\). Denoting by \(A^{<\omega_1}\) the collection of all
\(\alpha\)-sequences \((a_\beta)_{\beta<\alpha}\) in \(A\) for any ordinal \(\alpha<\omega_1\),
the axiom \({\sf DC}_{\omega_1}\) can be stated in the following way: if \(A\) is a non-empty set and
\(G\colon A^{<\omega_1}\twoheadrightarrow A\) is a multivalued map with non-empty values, then
there exists an \(\omega_1\)-sequence \((a_\alpha)_{\alpha<\omega_1}\) in \(A\) such that
\(a_\alpha\in G((a_\beta)_{\beta<\alpha})\) for every \(\alpha<\omega_1\). See
\cite{Jech,Gigli} for further details on this matter.
\begin{lemma}\label{lem:conseq_DC_omega_1}
Let \((\mathbb B,\|\cdot\|)\) be a Banach space and \(G\subseteq\mathbb B\setminus\{0\}\) a non-empty set.
Assume that for any \(v\in\mathbb B\setminus\{0\}\) there exists \(w\in G\) such that \(\|v\|=\|v-w\|+\|w\|\).
Then for any \(v\in\mathbb B\) there exists a sequence \((w_n)_{n\in\N}\subseteq G\cup\{0\}\)
such that \(\sum_{n=1}^\infty\|w_n\|=\|v\|\) and \(\sum_{n=1}^\infty w_n=v\).
\end{lemma}
\begin{proof}
Fix \(v\in\mathbb B\). Given \(u\in\mathbb B\), we call \(Y(u)\) the set of \(w\in G\) such
that \(\|u\|=\|u-w\|+\|w\|\). Clearly, \(Y(u)\neq\varnothing\) if and only if \(u\neq 0\).
We denote by \(\star\) a generic element not belonging to \(G\cup\{0\}\), and
we define \(E\colon(G\sqcup\{\star\})^{<\omega_1}\twoheadrightarrow(G\cup\{0\})\sqcup\{\star\}\) as
\[
E((w_\beta)_{\beta<\alpha})\coloneqq\left\{\begin{array}{lll}
\star\\
Y(v-\sum_{\beta<\alpha}w_\beta)\\
\{0\}
\end{array}\quad\begin{array}{lll}
\text{ if }\star\in(w_\beta)_{\beta<\alpha}\text{ or }\sum_{\beta<\alpha}\|w_\beta\|=\infty,\\
\text{ if }\sum_{\beta<\alpha}\|w_\beta\|<\infty\text{ and }\sum_{\beta<\alpha}w_\beta\neq v,\\
\text{ if }\sum_{\beta<\alpha}\|w_\beta\|<\infty\text{ and }\sum_{\beta<\alpha}w_\beta=v.
\end{array}\right.
\]
By the axiom \({\sf DC}_{\omega_1}\), there exists an \(\omega_1\)-sequence
\((w_\alpha)_{\alpha<\omega_1}\subseteq(G\cup\{0\})\sqcup\{\star\}\) such that
\[
w_\alpha\in E((w_\beta)_{\beta<\alpha})\quad\text{ for every }\alpha<\omega_1.
\]
We claim that for any \(\alpha<\omega_1\) it holds that
\begin{equation}\label{eq:transfin_ind_cond}
\star\notin\{w_\beta\}_{\beta<\alpha},\qquad\sum_{\beta<\alpha}\|w_\beta\|<\infty,\qquad
\|v\|=\bigg\|v-\sum_{\beta<\alpha}w_\beta\bigg\|+\sum_{\beta<\alpha}\|w_\beta\|.
\end{equation}
We prove it via a transfinite induction argument. Trivially, \eqref{eq:transfin_ind_cond} holds
for \(\alpha=0\). Next, fix \(\bar\alpha<\omega_1\) and assume \eqref{eq:transfin_ind_cond} holds
for all \(\alpha<\bar\alpha\). If \(\bar\alpha\) is a successor ordinal, say \(\bar\alpha=\gamma+1\), then:
\begin{itemize}
\item either \(\sum_{\beta<\gamma}\|w_\beta\|<\infty\) and \(\sum_{\beta<\gamma}w_\beta=v\), so that
\(w_\gamma=0\) and \eqref{eq:transfin_ind_cond} holds for \(\alpha=\bar\alpha\);
\item or \(\sum_{\beta<\gamma}\|w_\beta\|<\infty\) and \(u\coloneqq v-\sum_{\beta<\gamma}w_\beta\neq 0\),
so that \(w_\gamma\in Y(u)\) and thus
\[
\|v\|=\|u\|+\sum_{\beta<\gamma}\|w_\beta\|=\|u-w_\gamma\|+\|w_\gamma\|+\sum_{\beta<\gamma}\|w_\beta\|
=\bigg\|v-\sum_{\beta<\bar\alpha}w_\beta\bigg\|+\sum_{\beta<\bar\alpha}\|w_\beta\|,
\]
which implies that \eqref{eq:transfin_ind_cond} holds for \(\alpha=\bar\alpha\).
\end{itemize}
If, instead, \(\bar\alpha\) is a limit ordinal, then \(\bar\alpha=\bigcup_{\alpha<\bar\alpha}\alpha\)
and thus \(\star\notin\{w_\beta\}_{\beta<\bar\alpha}\). Since \(\bar\alpha\) is at most countable,
we find \((\alpha_n)_{n\in\N}\) such that \(\alpha_n\leq\alpha_{n+1}<\bar\alpha\) for all \(n\in\N\)
and \(\bar\alpha=\bigcup_{n\in\N}\alpha_n\). Hence,
\(\sum_{\beta<\bar\alpha}\|w_\beta\|=\lim_n\sum_{\beta<\alpha_n}\|w_\beta\|\leq\|v\|<\infty\),
thus \(\sum_{\beta<\bar\alpha}w_\beta=\lim_n\sum_{\beta<\alpha_n}w_\beta\in\mathbb B\) exists and
\[
\bigg\|v-\sum_{\beta<\bar\alpha}w_\beta\bigg\|+\sum_{\beta<\bar\alpha}\|w_\beta\|=
\lim_{n\to\infty}\bigg(\bigg\|v-\sum_{\beta<\alpha_n}w_\beta\bigg\|+\sum_{\beta<\alpha_n}\|w_\beta\|\bigg)=\|v\|.
\]
All in all, we have shown that \eqref{eq:transfin_ind_cond} holds for \(\alpha=\bar\alpha\). By transfinite induction,
we deduce that \eqref{eq:transfin_ind_cond} holds for every \(\alpha<\omega_1\), proving the claim. Therefore,
\(\{\alpha<\omega_1:\|w_\alpha\|\geq 1/k\}\) is finite for every \(k\in\N\), so that
\(A\coloneqq\{\alpha<\omega_1:w_\alpha\neq 0\}=\bigcup_{k\in\N}\{\alpha<\omega_1:\|w_\alpha\|\geq 1/k\}\)
is at most countable and accordingly \(\gamma\coloneqq\sup A+1<\omega_1\). Given that \(w_\gamma=0\),
we conclude that \(\sum_{\beta<\gamma}w_\beta=v\) and thus also \(\sum_{\beta<\gamma}\|w_\beta\|=\|v\|\).
This proves the statement.
\end{proof}

We are now in a position to prove our main result.
\begin{proof}[Proof of Theorem \ref{thm:main}]
Proposition \ref{prop:extract_monot} implies that for any \(f\in{\rm FV}(\X)\setminus\{0\}\) there exists a monotone
function \(g\in{\rm FV}(\X)\setminus\{0\}\) such that \(\|f\|_{{\rm FV}(\X)}=\|f-g\|_{{\rm FV}(\X)}+\|g\|_{{\rm FV}(\X)}\).
By applying Lemma \ref{lem:conseq_DC_omega_1} with \(\mathbb B\coloneqq{\rm FV}(\X)\) and
\(G\coloneqq\{g\in{\rm FV}(\X)\setminus\{0\}:g\text{ is monotone}\}\), we thus obtain that for any given
\(f\in{\rm FV}(\X)\) there exists a sequence of monotone functions \((f_n)_{n\in\N}\subseteq{\rm FV}(\X)\)
such that \(\sum_{n=1}^\infty\|f_n\|_{{\rm FV}(\X)}=\|f\|_{{\rm FV}(\X)}\) and \(\sum_{n=1}^\infty f_n=f\)
absolutely in \({\rm FV}(\X)\).
Thanks to \eqref{eq:equiv_norm_FV}, we deduce that
\(\sum_{n=1}^\infty f_n=f\) absolutely in \(L^{q/(q-1)}(\X)\). In particular
\(|Df|\leq\sum_{n=1}^\infty|Df_n|\), which together with \(\sum_{n=1}^\infty|Df_n|(\X)=|Df|(\X)\)
imply that \(|Df|=\sum_{n=1}^\infty|Df_n|\), as desired.
\end{proof}
\end{document}